\documentclass[reqno]{amsart} 
\usepackage{amssymb}
\usepackage{amsmath}
\usepackage{graphicx,color}
\usepackage{t1enc}
\usepackage[utf8]{inputenc}

\usepackage{amsthm}

\usepackage{enumitem}

\usepackage{latexsym}
\usepackage{subfigure}
\usepackage{tikz}
\usepackage{float}
\usepackage{graphicx,color}
\newtheorem{theorem}{Theorem}[section]

\newtheorem{lemma}{Lemma}[section]

\newtheorem{remark}{Remark}[section]

\newcommand\restr[2]{{
		\left.\kern-\nulldelimiterspace 
		#1 
		\vphantom{\big|} 
		\right|_{#2} 
}}

\begin{document}

\title[Quasilinear Schr\"odinger equation with singular nonlinearity]{About existence and regularity of positive solutions for a Quasilinear Schr\"odinger equation with singular nonlinearity}

\author{\bf\large Ricardo Lima Alves and Mariana Reis}\footnote{Ricardo Lima Alves acknowledges the support of CNPq/Brazil } \hspace{2mm}
 

\maketitle
\begin{center}
	{\bf\small Abstract}
	
	\vspace{3mm} \hspace{.05in}\parbox{4.5in} {{\small This paper deals with the existence of positive solution for the singular quasilinear Schr\"odinger equation
			$-\Delta u -\Delta (u^{2})u=h(x) u^{-\gamma} + f(x,u)~\mbox{in} ~ \Omega,$
			where $\gamma > 1$, $\Omega \subset \mathbb{R}^{N}, (N\geq 3)$ is a bounded smooth domain, $0<h\in L^{1}(\Omega)$, $f$ is a measurable function that can change signal and can be sublinear or has critical growth. Inspired by Sun \cite{Y} we derive a compatible condition on the couple $(h(x),\gamma)$, which is optimal for the existence of $H_{0}^{1}$-solution for this problem. } }
\end{center}


%
%
%
\section{\bf Introduction}
\def\theequation{1.\arabic{equation}}\makeatother
\setcounter{equation}{0}

In this article we are concerned with the existence of solution for the following quasilinear Schr\"odinger equation 
\begin{equation*}
(P)\left\{
\begin{array}{l}
-\Delta u -\Delta (u^{2})u=h(x) u^{-\gamma} + f(x,u)~\mbox{in} ~ \Omega,\\
u> 0~\mbox{in}~ \Omega,~~  u(x)=0~\mbox{on}~\partial \Omega,
\end{array}
\right.
\end{equation*}
where $1<\gamma, \Omega \subset \mathbb{R}^{N} (N\geq 3)$ is a bounded smooth domain, $0<h\in L^{1}(\Omega)$, $f:\Omega \times \mathbb{R}\longrightarrow \mathbb{R}$ is a measurable function and satisfies one of the following conditions
\begin{itemize}
	\item[$(f)_{1}$] $f(x,s)=b(x)s^{p}$ with $p\in (0,1), b\in L^{\infty}(\Omega)$ and $b^{+}\not\equiv 0$ in $\Omega$,
	\item[$(f)_{2}$] $f(x,s)=-b(x)s^{22^{\ast}-1}$ with $0\leq b\in L^{\infty}(\Omega)$.
\end{itemize}

Recently, some papers have worked with equation of the form
\begin{equation}\label{2}
-\Delta u -\Delta (u^{2})u=f(x,u)~\mbox{in} ~ \Omega,
\end{equation}
where $\Omega \subset \mathbb{R}^{N}$ is a bounded smooth domain and $f$ is non-singular. See for example \cite{AGS, LLL, LP, CRK} and its references, where the authors used variational methods to prove the existence of solution. In these works the nonlinearity is non-singular and therefore the functional energy associated to the problem has a good regularity to use the usual techniques for functional of class $C^{1}$.

When $f$ is singular, problems of type (\ref{2}) was studied by Do \'O-Moameni \cite{DM}, Liu-Liu-Zhao \cite{JDP} and Wang \cite{W}. In \cite{DM} the authors studied the problem
\begin{equation}\label{3}
-\Delta u -\frac{1}{2}\Delta (u^{2})u=\lambda |u|^{2}u-u-u^{-\gamma}~\mbox{in} ~ \Omega, u>0~\mbox{in} ~ \Omega~,
\end{equation}
where $\Omega$ is a ball in $\mathbb{R}^{N}$ centered at the origin, $0<\gamma<1$ and $N\geq 2$. They show, using the Nehari manifold method, that problem (\ref{3}) has a radially symmetric solution $u\in H_{0}^{1}(\Omega)$ for all $\lambda \in I$, where $I$ is a open and bounded interval. 

In 2016 Liu-Liu-Zhao in \cite{JDP} considered the problem
\begin{equation}\label{4}
-\Delta_{s} u -\frac{s}{2^{s-1}}\Delta (u^{2})u=h(x)u^{-\gamma}+\lambda u^{p}~\mbox{in} ~ \Omega, u>0~\mbox{in} ~ \Omega~,
\end{equation}
where $\Delta_{s}$ is the $s$-Laplacian operator, $s>1$,$\gamma>0$,  $\Omega\subset \mathbb{R}^{N}(N\geq 3)$ is a bounded smooth domain, $2s<p+1<\infty$ and $h(x) \geq 0$ is a nontrivial measurable function satisfying the following condition: there exists a function $\phi_{0}\geq 0$ in $C_{0}^{1}(\overline{\Omega})$ and $q>N$ such that $h(x)\phi_{0}^{-\gamma}\in L^{q}(\Omega)$. Combining the sub and supersolution method, truncation argument and variational methods, they proved the existence of $\lambda_{\ast}>0$ such that the problem (\ref{4}) has at least two solutions for $\lambda \in (0,\lambda_{\ast})$.

Recently Wang in \cite{W} proved the existence and uniqueness of solution to the following quasilinear Schr\"odinger equation
\begin{equation}
-\Delta u -\Delta (u^{2})u=h(x) u^{-\gamma}-u^{p-1}~\mbox{in} ~ \Omega, u>0~\mbox{in} ~ \Omega~,
\end{equation}
where $\Omega\subset \mathbb{R}^{N}(N\geq 3)$ is a bounded smooth domain, $\gamma \in (0,1), p\in [2,22^{\ast}]$ and $0<h\in L^{\frac{22^{\ast}}{22^{\ast}-1+\gamma}}(\Omega)$. The author used global minimization arguments to prove the existence of solution. Here after the use of variable change developed in Colin-Jeanjean \cite{CJ} the functional associated to the dual problem is well defined in $H^{1}_{0}(\Omega)$ and continuous.

Before stating our results we would like to cite here the work of Sun \cite{Y}. In this work the following problem was considered 
\begin{equation}\label{1}
-\Delta u=h(x)u^{-\gamma}+b(x)u^{p}~\mbox{in}~\Omega, ~u=0~\mbox{on}~\partial \Omega,
\end{equation}
where $\Omega\subset \mathbb{R}^{N}$ is a bounded smooth domain, $b\in L^{\infty}(\Omega)$ is a non-negative function, $0<p<1$, $\gamma > 1$ and $0<h\in L^{1}(\Omega)$. The author has proved, using variational methods, that the existence of positive solution in $H^{1}_{0}(\Omega)$ of the problem (\ref{1}) is related to a compatibility hypothesis between on the couple $(h(x),\gamma)$, more precisely it has been proved that the problem $(\ref{1})$ has a solution in $H^{1}_{0}(\Omega)$ if and only if there is $v_{0}\in H^{1}_{0}(\Omega)$ such that
\begin{equation}\label{C}
\int_{\Omega} h(x)|v_{0}|^{1-\gamma}<\infty.
\end{equation} 

The main difficulty there was because of the strong silgularity that causes a serious loss of regularity of the functional energy associated with which it is not continuous. In order to deal with these difficulty she worked with appropriate constrainsts sets to restore the integrability of singular term. Moreover, it was essential in its approach that nonlinearity was homogeneous.

Motivated by \cite{Y} a natural question arises: the hypothesis of compatibility between on the couple $(h(x),\gamma)$ given by (\ref{C}) remains necessary and sufficient for the existence of solution for our class of problems $(P)$? In this paper we will give a positive answer to this question. Also requesting more regularity in the function $h$ we prove that the solution have $C^{1,\alpha}$ regularity and as a consequence of this the solution is unique.

Our main results are 

\begin{theorem}\label{T1}
	Assume that $(f)_{1}$ is satisfied. Then: 
	
	\begin{itemize}
		\item[a)] the problem $(P)$ admits an solution $u\in H_{0}^{1}(\Omega)$ if and only if there exists $v_{0}\in H_{0}^{1}(\Omega)$ such that (\ref{C}) is satisfied.
		
		\item[b)] if $b\geq 0$ and there exist constants $c>0$ and $\beta \in (0,1)$ such that
		\begin{equation}\label{D}
		h(x)\leq c d^{\gamma-\beta}(x,\partial \Omega), \forall x\in \Omega,
		\end{equation}
		then the solution $u$ given in $a)$ belongs to $C^{1,\alpha}(\overline{\Omega})$ for some $\alpha \in (0,1)$. In particular the problem $(P)$ has a unique solution in $H_{0}^{1}(\Omega)$.
	\end{itemize}
\end{theorem}

\begin{theorem}\label{T2}
	If $(f)_{2}$ is satisfied, the problem $(P)$ admits an unique solution $ u\in H_{0}^{1}(\Omega)$ if and only if there exists $v_{0}\in H_{0}^{1}(\Omega)$ such that (\ref{C}) is satisfied.
	
\end{theorem}

To prove our main result let us use the method of changing variables developed by Colin-Jeanjean \cite{CJ}. After this the functional associated with the dual problem is not homogeneous and this causes several difficulties. For example, the techniques used by previous work do not apply directly. To cover this difficulty we will make a  careful analysis of the fiber maps associated to the functional of the dual problem and will approach the problem in a new way to prove the existence of a solution to the problem $(P)$.

Now let us mention some contributions from our work. In this work we consider the most general potentials, for instance the potential $ b $ can change signal. The regularity of solution (and hence uniqueness) for problem with strong singularity has not been treated yet. Theorem \ref{T2} completes the study made by Wang in \cite{W} in the sense that we now consider the case $\gamma>1$, while \cite{W} consider the case $0<\gamma<1$. Moreover in our work we consider the more general potentials also.

This paper is organized as follows: In the next section we reformulate the problem $(P)$ into a new one which finds its natural setting in the Sobolev space $H_{0}^{1}(\Omega)$ and we will present some preliminary lemmas. In section $3$, we give the proof of Theorem \ref{T1} and in section $4$ the proof of Theorem \ref{T2}. The last section consists of an appendix to which we will study the problem $(P)$ with $b(x)\equiv \lambda b(x), \lambda \geq 0$ and prove that the solutions found in the Theorem \ref{T1} vary continuously with respect to $\lambda$ and the norm from $H_{0}^{1}(\Omega)$.

$Notation$. In the rest of the paper we make use of the following notation:
\begin{itemize}[leftmargin=*]  
\item $c,C$ denote positive constants, which may vary from line to line,
\item $H_{0}^{1}(\Omega)$ denote the Sobolev Space equipped with the norm $||u|| \!\!=\!\! \left(\displaystyle\int_{\Omega}|\nabla u|^{2}dx\right)^{2}$,
\item $L^{s}(\Omega)$ denotes the Lebesgue Space with the norms $||u||_{s}= \left(\displaystyle\int_{\Omega}|\nabla u|^{s}dx\right)^{1/s}$, for $1\leq p<\infty$ and $||u||_{\infty}=\inf \left\{C>0: |u(x)|\leq C~\mbox{a.s. in}~ \Omega\right\}$,
\item for each set $B\subset \mathbb{R}^{N}$ the characteristic function of $B$ is denoted by $\chi_{B}$.  
\end{itemize}

\section{Variational framework and Preliminary Lemmas}

In this section we provide preliminary results wich will be used to prove the existence of a solution of the problem $(P)$. By solutions we mean here weak solutions in $H_{0}^{1}(\Omega)$, that is $u \in H_{0}^{1}(\Omega)$ satisfying $u(x) > 0$, in $\Omega$ and
$$
\displaystyle \int_{\Omega} [(1+u^{2})\nabla u \nabla \varphi+2u\vert \nabla u\vert^{2}\varphi - h(x)u^{-\gamma} \varphi-f(x,u)\varphi]dx=0,
$$
for every $\varphi \in H^{1}_{0}(\Omega)$, which is formally the variational formulation of the following functional 
$J:D(J)\subset H_{0}^{1}(\Omega)\rightarrow \mathbb{R}$
$$J(u)=\frac{1}{2}\displaystyle\int_{\Omega} (1+2u^{2})|\nabla u|^{2}+\frac{1}{\gamma-1}\int_{\Omega} h(x)|u|^{1-\gamma}-\int_{\Omega} F(x,u),$$
where $D(J)=\left\{u\in H_{0}^{1}(\Omega): \displaystyle\int_{\Omega} h(x)|u|^{1-\gamma}<\infty \right\}$ if $D(J)\neq \emptyset$ and $F(x,s)=\displaystyle\int_{0}^{s}\!\!\!f(x,t)dt$.

However, this functional is not well-defined, because $\displaystyle\int_{\Omega} u^{2}|\nabla u|^{2}dx$ is not finite for all $u\in H^{1}_{0}(\Omega) $, hence it is difficult to apply variational methods directly. Firstly we use the method developed in \cite{CJ} introducing the unknown variable $v:=g^{-1}(u),$ where $g$ is defined by
$$g^{\prime}(t)=\frac{1}{\left(1+2|g(t)|^{2}\right)^{\frac{1}{2}}},~\forall t\in [0,\infty), ~ 
g(t)=-g(-t),~\forall t\in (-\infty,0].$$

It is easy to see that if $v$ is solution of
\begin{equation*} 
(P_{A})\left\{
\begin{array}{l}
-\Delta v=\left[ h(x) (g(v))^{-\gamma} + f(x,g(v))\right]g^{\prime}(v) ~\mbox{in} ~ \Omega,\\
v> 0~\mbox{in}~ \Omega,~~  v(x)=0 ~\mbox{on}~ \partial \Omega,
\end{array}
\right.
\end{equation*}
if and only if $u = g(v)$ is solution of $(P)$. We will call the problem $(P_{A})$ of dual problem to $(P)$.

The weak form of the equation $(P_A)$ is
$$
\int_{\Omega} \nabla v \nabla \phi dx=\int_{\Omega}  h(x) (g(v))^{-\gamma}g^{\prime}(v)\phi dx+\int_{\Omega} f(x,g(v))g^{\prime}(v)\phi dx,
$$
for every $\phi \in H_{0}^{1}(\Omega)$ and therefore $v$ is a critical point of functional
$$\Phi(v)=\frac{1}{2}\int_{\Omega} |\nabla v|^{2}+\frac{1}{\gamma-1}\int_{\Omega} h(x)|g(v)|^{1-\gamma}-\int_{\Omega} F(x,g(v)),$$
which is defined in $D(\Phi)=\left\{v\in H_{0}^{1}(\Omega): \displaystyle\int_{\Omega} h(x)|g(v)|^{1-\gamma}<\infty \right\}$ if $D(\Phi)\neq \emptyset$ and $F(x,s)=\displaystyle\int_{0}^{s}f(x,t)dt$.
Now, we list some properties of $g$, whose proofs can be found in Liu \cite{L}.
\begin{lemma}\label{L1} The function $g$ satisfies the following properties:
	\begin{itemize}
		\item[(1)] $g$ is uniquely defined, $C^{\infty}$ and invertible;
		\item[(2)] $g(0)=0$;
		\item[(3)] $0<g^{\prime}(t)\leq 1$ for all $t\in \mathbb{R}$;
		\item[(4)] $\frac{1}{2}g(t)\leq tg^{\prime}(t)\leq g(t)$ for all $t>0$;
		\item[(5)] $|g(t)|\leq |t|$ for all $t\in \mathbb{R}$;
		\item[(6)] $|g(t)|\leq K_{0}|t|^{\frac{1}{2}}$ for all $t\in \mathbb{R}$;
		\item[(7)] $(g(t))^{2}-g(t)g^{\prime}(t)t\geq 0$ for all $t\in \mathbb{R}$;
		\item[(8)] There exists a positive constant $C$ such that $|g(t)|\geq C|t|$ for $|t|\leq 1$ and $|g(t)|\geq C|t|^{\frac{1}{2}}$ for all $|t|>1$;
		\item[(9)] $g^{\prime \prime}(t)<0$ when $t>0$ and $g^{\prime \prime}(t)>0$ when $t<0$;
		\item[(10)] the functions $(g(t))^{1-\gamma}$ and $(g(t))^{-\gamma}$ are decreasing for all $t>0$;
      	\item[(11)] the function $(g(t))^{p}t^{-1}$ is decreasing for all $t>0$;
      	\item[(12)] $|g(t)g^{\prime}(t)|<1/ \sqrt[]{2}$ for all $t\in \mathbb{R}$.
      \end{itemize}
\end{lemma}
\begin{proof}
	We only prove $(10),(11)$. Since  $g(t),g^{\prime}(t)>0 $ for each $t>0$ and $\gamma>1$ follows that 
	$$\left[(g(t))^{1-\gamma}\right]^{\prime}=(1-\gamma)(g(t))^{-\gamma}g^{\prime}(t) < 0, ~\forall t>0.$$
        Hence the function $g^{1-\gamma}:(0,\infty)\longrightarrow \mathbb{R}$ is decreasing. Similarly we have that the function $g^{-\gamma}:(0,\infty)\longrightarrow \mathbb{R}$ is decreasing. 
        
        Let us prove $(11)$. To do this, note that
       \begin{align*}
        \left[(g(t))^{p}t^{-1}\right]^{\prime}=&p(g(t))^{p-1}g^{\prime}(t)t^{-1}-(g(t))^{p}t^{-2}\\
        =&p(g(t))^{p-1}(g^{\prime}(t)t)t^{-2}-(g(t))^{p}t^{-2}\\ 
        <& t^{-2}\left[(g(t))^{p-1}g(t)-(g(t))^{p}\right]<0,
       \end{align*}
       where we use the item $(4)$ of this Lemma and $p<1$. Therefore the function $(g(t))^{p}t^{-1}$ is decreasing for all $t>0$.
    \end{proof}

The next lemma gives us a relation of duality between the compatibility hypothesis for the problems $(P)$ and $(P_{A})$.
\begin{lemma}\label{L2} Let $v>0$ in $\Omega$.  The following conditions are equivalent:
	\begin{itemize}
		\item[(a)] $\displaystyle\int_{\Omega} h(x)|v|^{1-\gamma}<\infty$;
		\item[(b)] $\displaystyle\int_{\Omega} h(x)(g(v))^{-\gamma}g^{\prime}(v)v<\infty$;
		\item[(c)] $\displaystyle\int_{\Omega} h(x)(g(v))^{1-\gamma}<\infty$.
	\end{itemize}
	
\end{lemma}
\begin{proof}
	
Firstly let us prove that	$(a)\Rightarrow (b)$. We denote by $A_{1}=\left\{x\in \Omega:|v(x)|\leq 1\right\}$ and $A_{2}=\left\{x\in \Omega:|v(x)|> 1\right\}$. By the Lemma \ref{L1} $(4),(8)$ we have
	$$|h(x)(g(v))^{-\gamma}g^{\prime}(v)v|\leq C^{1-\gamma} h(x)|v|^{1-\gamma},~\forall x\in A_{1},$$
    and
    $$
    \begin{array}{rl}
|h(x)(g(v))^{-\gamma}g^{\prime}(v)v| \leq & h(x)|(g(v))^{1-\gamma}|\\
                                     \leq & C^{1-\gamma}h(x)|v|^{\frac{1-\gamma}{2}}\\
                                     \leq & C^{1-\gamma}h(x),~\forall x\in A_{2},
    \end{array}
    $$
	and this implies that
	\begin{equation}\label{100}
	h(g(v))^{-\gamma}g^{\prime}(v)v\in L^{1}(A_{1})\cap L^{1}(A_{2}),
	\end{equation}
	because $h|v|^{1-\gamma},h\in L^{1}(\Omega)$.
	
	 Now, using \eqref{100} we conclude that $h(g(v))^{-\gamma}g^{\prime}(v)v\in L^{1}(\Omega)$ because
	$$h(x)(g(v))^{-\gamma}g^{\prime}(v)v=h(x)(g(v))^{-\gamma}g^{\prime}(v)v\chi_{A_{1}}+h(x)(g(v))^{-\gamma}g^{\prime}(v)v\chi_{A_{2}}.$$
	
	To prove that $(b)\Rightarrow (c)$ note that by Lemma \ref{L1} $(4)$,
	$$\frac{1}{2}\int_{\Omega} h(x)(g(v))^{1-\gamma}\leq \int_{\Omega} h(x)(g(v))^{-\gamma}g^{\prime}(v)v <\infty. $$
	
	Finally to prove that $(c)\Rightarrow (a)$ we use the Lemma \ref{L1} $(5)$ to obtain that
	$$\int_{\Omega} h(x)|v|^{1-\gamma}\leq \int_{\Omega} h(x)(g(v))^{1-\gamma}<\infty, $$
	and the proof is completed.
\end{proof}

Note that if $v_{0}$ satisfies the condintion (\ref{C}) then, since $\rvert v_{0}\rvert \in H_{0}^{1}(\Omega)$ we have that $\rvert v_{0}\rvert$ satisfies the condition (\ref{C}). Hence we may assume that $v_{0}\geq 0$. Also as we are interested in positive solution let us work on the following subset of $H_{0}^{1}(\Omega)$
$$V_{+}=\left\{v\in H_{0}^{1}(\Omega)\setminus\{0\}: v\geq 0\right\}.$$

Assume that $v\in V_{+}$ and 
\begin{equation}\label{CD}
\int_{\Omega} h(x)|v|^{1-\gamma}<\infty (~\mbox{and therefore}~ \int_{\Omega} h(x)(g(v))^{1-\gamma}< \infty),
\end{equation}
and consider the fiber map $\phi_{v}:(0,\infty)\rightarrow \mathbb{R}$
$$\phi_{v}(t):=\Phi(tv)=\frac{t^{2}}{2}\int_{\Omega} |\nabla v|^{2}+\frac{1}{\gamma-1}\int_{\Omega} h(x)(g(tv))^{1-\gamma}-\int_{\Omega} F(x,g(tv)).$$

The understanding of the fibering maps will be extremely important in the next section. Let us start by proving that for each $v$ satisfying (\ref{CD}) the fiber map associated to $v$ has a good regularity. 

\begin{lemma} \label{C1}
	We have that $\phi_{v}\in C^{1}((0,\infty),\mathbb{R})$ for each $v$ satisfying (\ref{CD}).
\end{lemma}
\begin{proof}
	We have that prove just that  $\Gamma:(0,\infty )\longrightarrow \mathbb{R}$ defined by 
	$$\Gamma (t)= \displaystyle\int_{\Omega} h(x)(g(tv))^{1-\gamma },$$
	is of classe $C^{1}$. To do this we fix $t>0$ and note that for each $s>0$ by Mean Value Theorem there exits a mensurable function $\theta=\theta(s,x) \in (0,1)$ such that, 
	$$\Gamma(t+s)-\Gamma(t)=(1-\gamma)\displaystyle \int_{\Omega} h(x)(g((t+\theta s)v))^{-\gamma}g^{\prime}((t+\theta s)v) sv$$
	and $t+\theta(s,x)s\longrightarrow t$ as $s\longrightarrow 0$. 
	
	As the function $(g(t))^{-\gamma}g^{\prime}(t),t>0$ is decreasing (by Lemma \ref{L1}$(9),(10)$) follows that $(g((t+\theta s)v))^{-\gamma}g^{\prime}((t+\theta s)v)\leq (g(tv))^{-\gamma}g^{\prime}(tv)$. Furthermore, as consequence of the Lemma \ref{L2}, $h(g(tv))^{-\gamma}g^{\prime}(tv)v\in L^{1}(\Omega)$. Hence we are able to apply the Lebesgue's dominated convergence theorem to infer that
	$$\Gamma^{\prime}(t)=\displaystyle \lim_{s\to 0}\frac{\Gamma(t+s)-\Gamma(t)}{s}=(1-\gamma)\displaystyle \int_{\Omega} h(x)(g(tv))^{-\gamma}g^{\prime}(tv)v,$$
	that is, the derivative $\Gamma^{\prime}(t)$ there exists for all $t>0$ and is given by the last expression above. Now, using the Lemma \ref{L2} and the Lebesgue's dominated convergence theorem we have that $\Gamma^{\prime}:(0,\infty)\longrightarrow \mathbb{R}$ is a continuous function.
	\end{proof}

The next lemma guarantees that for each $v$ satisfying (\ref{CD}) the fiber map $\phi_{v}$ assumes its minimum value and therefore $\phi_{v}$ has a critical point.

\begin{lemma}\label{L3} For each $v$ satisfying (\ref{CD}) there holds
	$$\displaystyle \lim_{t\to 0}\phi_{v}(t)=\infty~\mbox{and}~\displaystyle \lim_{t\to \infty}\phi_{v}(t)=\infty,$$
	and therefore there exists $t(v)>0$ such that
	$$\phi_{v}(t(v))=\displaystyle \inf_{t>0}\phi_{v}(t).$$
	\begin{proof}
		Firstly we will consider the sublinear case, that is $(f)_{1}$ with $p\in (0,1)$. By the Lemma \ref{L1} $(5)$ we have that
		$$\int_{\Omega} h(x) (g(tv))^{1-\gamma}dx\geq t^{1-\gamma}\int_{\Omega} h(x) |v|^{1-\gamma},$$
		and
		$$ t^{p+1}\int_{\Omega} |b(x)||v|^{p+1}\geq  \rvert \int_{\Omega} b(x)(g(tv))^{p+1}\rvert \geq 0,$$
		which implies that
		$$\displaystyle \lim_{t\to 0}\int_{\Omega} h(x) (g(tv))^{1-\gamma}dx=\infty ~\mbox{and}~\displaystyle \lim_{t\to 0}\int_{\Omega} b(x)(g(tv))^{p+1}=0.$$
		
		Since $\gamma>1$ we have that $\displaystyle \lim_{t\to 0}\phi_{v}(t)=\infty$.
		
		By other side
		$$\displaystyle \lim_{t\to \infty}\phi_{v}(t)\geq  \lim_{t\to \infty}t^{2}\left[ ||v||^{2}-t^{p-2}\frac{\rVert b\rVert_{\infty}}{p+1}\int_{\Omega} |v|^{p+1}dx\right]=\infty.$$
		
		The continuity of the function $\phi_{v}$ and the limits $\displaystyle \lim_{t\to 0}\phi_{v}(t)=\infty$ and $\displaystyle \lim_{t\to \infty}\phi_{v}(t)=\infty$ implies that there exists $t(v)>0$ such that $\phi_{v}(t(v))=\displaystyle \inf_{t>0}\phi_{v}(t).$
		
		If $(f)_{2}$ is satisfied the proof is similar.
	\end{proof}
	
	The following picture give the possible graph of the fiber map.
\end{lemma}
\begin{figure}[h]
	\centering
	\begin{tikzpicture}[scale=.30]
	\draw[thick, ->] (-1, 3) -- (6, 3);
	\draw[thick, ->] (0, 0) -- (0, 7);
	\draw[thick] (0.5, 6.8) .. controls (0, 1) and (4, -2) ..(5,6.8);
	\draw (0,7) node[left]{$\phi_{v}$};
	\draw (6,3) node[below]{$t$};
	\draw (0, 3) node[below left]{$0$};
	\draw (2.5,2.8) node[above]{$t(v)$};
	\draw[thick, dotted] (2.5, 3) -- (2.5,1.3);
	\draw (3,0) node[below]{\textrm{Fig. 1 }};
	\end{tikzpicture}
	\hspace{0.8cm}
	\begin{tikzpicture}[scale=.30]
	\draw[thick, ->] (-1, 3) -- (6, 3);
	\draw[thick, ->] (0, 0) -- (0, 7);
	\draw[thick] (0.5, 6.8) .. controls (1, 4) and (3.7, 2) ..(5.5,6.8);
	\draw (0,7) node[left]{$\phi_{v}$};
	\draw (6,3) node[below]{$t$};
	\draw (0, 3) node[below left]{$0$};
	\draw (3.5,1) node[above]{$t(v)$};
	\draw[thick, dotted] (2.9, 3) -- (2.9,4);
	\draw (3,0) node[below]{\textrm{Fig. 2 }};
	\end{tikzpicture}
\end{figure}
\begin{remark}
When $(f)_{1}$ is satisfied the graph of $\phi_{v}$ can be as in Figures $1$ and $2$. On the other hand if $(f)_{2}$ is satisfied then the graph of $\phi_{v}$ can only be as in Figure $2$.
\end{remark}

Motivated by \cite{Y} we define the following constraint sets for the problem $(P_{A})$ 
$$\mathcal{N}_{1}=\left\{v\in V_{+}: ||v||^{2}-\int_{\Omega} f(x,g(v))g^{\prime}(v)v\geq \int_{\Omega} h(x)(g(v))^{-\gamma}g^{\prime}(v)v \right\},$$
and
$$\mathcal{N}_{2}=\left\{v\in V_{+}: ||v||^{2}-\int_{\Omega} f(x,g(v))g^{\prime}(v)v= \int_{\Omega} h(x)(g(v))^{-\gamma}g^{\prime}(v)v \right\}.$$

We note that if $v$ is a solution to the problem $(P_{A})$ then $v\in \mathcal{N}_{2}$.

It should be noted that for $\gamma>1$, $\mathcal{N}_{2}$ is not closed as usual (certainly not weakly closed).

Next lemma ensures that any function $v\in V_{+}$ satisfying the following condition of dual compatibility  
\begin{equation}\label{E1}
\int_{\Omega} h(x) (g(v))^{1-\gamma}<\infty,~v\in V_{+},
\end{equation}
can be projected over $\mathcal{N}_{2}$. 
\begin{lemma}\label{L5}
	Assume that there exists $v\in V_{+}$ such that $(\ref{E1})$ is satisfied. Then there exists $t(v)>0$ such that $t(v)v\in \mathcal{N}_{2}$. 
\end{lemma}
\begin{proof}
	As $v$ satisfied \eqref{E1} it follows from the Lemma \ref{L2} that $v$ satisfied \eqref{CD} also and therefore by Lemma \ref{C1} we have that $\phi_{v}\in C^{1}((0,\infty),\mathbb{R})$. Follows from the Lemma \ref{L3} that there exists $t(v)>0$ such that
	$$\phi_{v}(t(v))=\displaystyle \inf_{t>0}\phi_{v}(t),$$
	that is $t(v)$ is a critical point of $\phi_{v}$ and hence $\phi_{v}^{\prime}(t(v))=0$, which implies that
	\begin{align*}
	||t(v)v||^{2}-\int_{\Omega} h(x) (g(t(v)v))^{-\gamma}g^{\prime}(t(v)v)(t(v)v)-\int_{\Omega} f(x,t(v)v)g^{\prime}(t(v)v)(t(v)v)\\=t(v)\phi_{v}^{\prime}(t(v))=0.
	\end{align*}
	
	So $t(v)v\in \mathcal{N}_{2}\subset \mathcal{N}_{1}$.

\end{proof}

The following lemmas will be used to prove the regularity of the solution.
\begin{lemma}\label{BN}
	 Let $\Omega$ be a bounded domain in $\mathbb{R}^{N}$ with smooth boundary $\partial \Omega$. Let $u\in L^{1}_{loc}(\Omega)$ and assume that, for some $k\geq 0$, $u$ satisfies, in the sense of distributions,
	$$
	\left\{
	\begin{array}{c}
	-\Delta u+ku\geq 0~\mbox{in}~\Omega\\
	u\geq 0~~\mbox{in}~~\Omega.\\
	\end{array}
	\right.
	$$
	Then either $u\equiv 0$, or there exists $\epsilon>0$ such that $u(x)\geq \epsilon d(x,\partial \Omega),~x\in \Omega.$
\end{lemma}
\begin{proof}
	See Brezis-Nirenberg [\cite{BN}, Theorem 3].
\end{proof}
\begin{lemma}\label{Ht}
	Let $a\in L^{1}(\Omega)$ and suppose that there exist constants $\delta \in (0,1)$ and $C>0$ such that $|a(x)|\leq C\phi_{1}^{-\delta}(x),$ for a.e. $x\in \Omega$. Then, the problem
	$$
	\left\{
	\begin{array}{c}
	-\Delta u= a ~\mbox{in}~\Omega\\
	u=0 ~~\mbox{on}~~\partial\Omega,\\
	\end{array}
	\right.
	$$
	has a unique solution $u\in H_{0}^{1}(\Omega)$. Furthermore, there exist constants $\alpha \in (0,1)$ and $M>0$ depending only on $C,\alpha, \Omega $ such that $u\in C^{1,\alpha}(\overline{\Omega})$ and $|u|_{1,\alpha}<M$.
\end{lemma}
\begin{proof}
	See Hai [\cite{H}, Lemma 2.1, Remark 2.2].
\end{proof}
\begin{remark} For a later use we recall that there exist constants $l_{1},l_{2}>0$ such that
	$$l_{1}d(x,\partial \Omega)\leq \phi_{1}(x)\leq l_{2}d(x,\partial \Omega),~x\in \Omega,$$
	where $\phi_{1}$ is the first eigenfunction of $(-\Delta,H_{0}^{1}(\Omega))$.
\end{remark}

\begin{lemma}\label{M}
	Let $\psi_{j}:\Omega \times (0,\infty)\longrightarrow [0,\infty), j=1,2$ are measurable functions such that
	$$\psi_{1}(x,s)\leq \psi_{2}(x,s)~\mbox{for all}~(x,s)\in \Omega \times (0,\infty),$$
	and for each $x\in \Omega$, the function $s\longmapsto \psi_{1}(x,s)s^{-1}$ is decreasing on $(0,\infty).$
	Furthermore let $u,v\in H^{1}(\Omega)$, with $u\in L^{\infty}(\Omega), u>0, v>0$ on $\Omega$ are such that
	$$-\Delta u\leq \psi_{1}(x,u)~\mbox{and}~ -\Delta v\geq \psi_{2}(x,v)~\mbox{on}~\Omega.$$
	If $u\leq v$ on $\partial \Omega$ and $\psi_{1}(\cdot,u)$ (or $\psi_{2}(\cdot,u)$) belongs to $L^{1}(\Omega)$, then $u\leq v$ on $\Omega$.
\end{lemma}
\begin{proof}
	See Mohammed [\cite{M}, Theorem 4.1].
\end{proof}

\section{Proof of Theorem \ref{T1}.}

In this section we will show the Theorem \ref{T1}. First we wil given some preliminary lemmas.
\begin{lemma}\label{LL1} The set $\mathcal{N}_{1}\neq \emptyset$  and the functional $\Phi$ is coercive in $\mathcal{N}_{1}$.
\end{lemma}
\begin{proof}
	Since \eqref{C} is satisfied it follows from Lemmas \ref{L2}, \ref{L5} that $\mathcal{N}_{1}\neq \emptyset$. Now, let us prove that $\Phi$ is coercive. Indeed for every $v\in \mathcal{N}_{1}$,
	$$
	\begin{array}{rl}
	\Phi(v)= & \displaystyle\frac{1}{2}\int_{\Omega} |\nabla v|^{2}+\frac{1}{\gamma-1}\int_{\Omega} h(x)(g(v))^{1-\gamma}-\frac{1}{p+1}\int_{\Omega} b(x)(g(v))^{p+1}\\
	\geq & \displaystyle\frac{1}{2}\int_{\Omega} |\nabla v|^{2}-\frac{\rVert b\rVert_{\infty}}{p+1}\int_{\Omega} (g(v))^{p+1},
	\end{array}
	$$
	and by Lemma \ref{L1} $(5)$ and Sobolev embedding we have
	$$\Phi(v)\geq \frac{1}{2}\int_{\Omega} |\nabla v|^{2}-\frac{\rVert b\rVert_{\infty}}{p+1}\int_{\Omega} |v|^{p+1}\geq \frac{||v||^{2}}{2}-C\frac{||v||^{p+1}}{p+1},$$
	for some constant $C>0$. Since $p\in (0,1)$ follows that $\Phi$ is coercive.
\end{proof}

By the Lemma \ref{LL1} we have that
$$J_{1}=\displaystyle \inf_{v\in \mathcal{N}_{1}}\Phi(v)~\mbox{and}~J_{2}=\displaystyle \inf_{v\in \mathcal{N}_{2}}\Phi(v),$$
are well defined with $J_{1},J_{2}\in \mathbb{R}$  and $J_{2}\geq J_{1}$.
\begin{lemma}\label{LL2}
	There exists $v\in \mathcal{N}_{2}$ such that $J_{1}=\Phi(v)=J_{2}.$
\end{lemma}
\begin{proof}
	By Lemma \ref{LL1} there exists a sequence $\left\{v_{n}\right\}\subset \mathcal{N}_{1}$ such that $\Phi(v_{n})\longrightarrow J_{1},$ and may assume that $\left\{v_{n}\right\}$ is bounded and exist $v \in H_{0}^{1}(\Omega)$ such that
	$$
	\left\{
	\begin{array}{l}
	v_{n}\rightharpoonup v~\mbox{in}~H_{0}^{1}(\Omega),\\
	v_{n}\longrightarrow v~\mbox{in}~L^{s}(\Omega)~\mbox{for all}~s\in (0,2^{\ast}),\\
	v_{n}\longrightarrow v~a.s.~\Omega.\\
	\end{array}
	\right.
	$$
	
	Since that $v_{n}>0$ follows that $v\geq 0$ in $\Omega$. Moreover, there exists a constant $C>0$ such that $||v_{n}||\leq C$ for every $n\in \mathbb{N}$. By definition of $\mathcal{N}_{1}$ and Lemma \ref{L1} $(3),(4),(5)$ we have that
	
	$$
	\begin{array}{rl}
	\dfrac{1}{2} & \displaystyle\int_{\Omega} h(x)(g(v_{n}))^{1-\gamma}  \leq \displaystyle\int_{\Omega} h(x)(g(v_{n}))^{-\gamma}g^{\prime}(v_{n})v_{n} \\
	& \leq  ||v_{n}||^{2}-\displaystyle\int_{\Omega} b(x)(g(v_{n}))^{p}g^{\prime}(v_{n})v_{n}\\
	& \leq  ||v_{n}||^{2}+\displaystyle\int_{\Omega} |b(x)||v_{n}|^{p+1}\\ 
	& \leq  ||v_{n}||^{2}+c||v_{n}||^{p+1} \leq  C^{2}+cC^{p+1}:=C,
	\end{array}
	$$
	where we used Sobolev embedding. Therefore using the Fatou's lemma in the last inequality we have that $\displaystyle\int_{\Omega} \theta(x)\leq C<\infty,$
	where 
	$$
	\theta(x)=\left\{
	\begin{array}{ccc}
	h(x)(g(v(x)))^{1-\gamma}, & \mbox{if} & v(x)\neq 0 \\
	\infty, & \mbox{if} & v(x)=0.\\
	\end{array}
	\right.
	$$
	Since that $g(0)=0$ (by Lemma \ref{L1} $(2)$) and $\displaystyle\int_{\Omega} \theta(x)<\infty$ follows that $v>0$ in $\Omega$. Once again by Fatou's lemma we obtain
	$$0<\int_{\Omega} h(x)(g(v))^{-\gamma}g^{\prime}(v)v\leq C< \infty.$$
	
	Consequently by the Lemma \ref{L3} and Lemma \ref{L5} there exists $t(v)>0$ such that $\phi_{v}(t(v))=\displaystyle \inf_{t>0} \phi_{v}(t)$ and $t(v)v\in \mathcal{N}_{2}$. Taking advantage of this information it follows that
	$$
	\begin{array}{l}
		J_{1} =  \displaystyle\lim_{n\to \infty} \Phi(v_{n})=  \displaystyle \liminf_{n\to \infty} \Phi(v_{n})\\
	                      =  \displaystyle \liminf_{n\to \infty} \left[ \displaystyle\frac{1}{2}\int_{\Omega} |\nabla v_{n}|^{2}+\frac{1}{\gamma-1}\int_{\Omega} h(x)g(v_{n})^{1-\gamma}-\frac{1}{p+1}\int_{\Omega} b(x)(g(v_{n}))^{p+1}\right]\\
	                    \geq  \displaystyle \liminf_{n\to \infty} \left[\displaystyle\frac{1}{2}\int_{\Omega} |\nabla v_{n}|^{2}\right]+\displaystyle \liminf_{n\to \infty} \left[\frac{1}{\gamma-1}\int_{\Omega} h(x)(g(v_{n}))^{1-\gamma}\right]-\frac{1}{p+1}\int_{\Omega}b(x) (g(v))^{p+1}\\
                     	\geq  \displaystyle\frac{1}{2}\int_{\Omega} |\nabla v|^{2}+\frac{1}{\gamma-1}\int_{\Omega} h(x)(g(v))^{1-\gamma}-\frac{1}{p+1}\int_{\Omega} b(x)(g(v))^{p+1}=\phi_{v}(1)\\
	                    \geq  \phi_{v}(t(v))=\Phi(t(v)v)\geq J_{2}\geq J_{1},
	                  \end{array}
                     	$$
	which implies that
	$$J_{1}=\phi_{v}(1)=\Phi(v)=J_{2},$$
	and $\phi_{v}(1)=\phi_{v}(t(v))= \displaystyle\inf_{t>0}\phi_{v}(t)$.  Hence $\phi^{\prime}_{v}(1)=0$ and consequently $v\in \mathcal{N}_{2}\subset \mathcal{N}_{1}$.
\end{proof}
Now let us prove the Theorem \ref{T1}.
\begin{proof} Firstly we will to prove $a)$. Suppose that $(P)$ has a solution $v_ {0}$. Taking $v_{0}$ as test function is easy to see that \eqref{C} is satisfied.
	
	 Now assume that \eqref{C} is satisfied. Let $v$ as in the Lemma \ref{LL2}. Let us prove that $v$ is a solution of problem $(P_{A})$. Consider $\varphi \in H_{0}^{1}(\Omega)$ such that $\varphi\geq 0$ in $\Omega$ and let $\epsilon>0$. By the Lemma \ref{L1} $(10)$
	$$\int_{\Omega}h(x) (g(v+\epsilon \varphi))^{1-\gamma}\leq \int_{\Omega} h(x)(g(v))^{1-\gamma}<\infty,$$
	and consequently by Lemmas \ref{L3}, \ref{L5} there exist $t(\epsilon)>0$ such that $\phi_{v+\epsilon \varphi}(t(\epsilon))	=\displaystyle\inf_{t>0}\phi_{v+\epsilon \varphi}(t)$ and $t(\epsilon)(v+\epsilon \varphi)\in \mathcal{N}_{2}$. Therefore
	$$\Phi(v+\epsilon\varphi)=\phi_{v+\epsilon \varphi}(1)\geq \phi_{v+\epsilon \varphi}(t(\epsilon))=\Phi(t(\epsilon)(v+\epsilon\varphi))\geq J_{2}=\Phi(v),$$
which implies that
	\begin{align}
\nonumber	&\int_{\Omega}\frac{h(x)(g(v+\epsilon \varphi))^{1-\gamma}-h(x)(g(v))^{1-\gamma}}{1-\gamma}\\
\nonumber	&\leq \frac{||v+\epsilon \varphi ||^{2}-||v||^{2}}{2}-\int_{\Omega}\frac{b(x)(g(v+\epsilon \varphi))^{p+1}-b(x)(g(v))^{p+1}}{p+1}.
	\end{align}
	
	Thus, dividing the last inequality by $\epsilon>0$ and passing to the $\liminf$ as $\epsilon\longrightarrow 0$, by Fatou's Lemma we have
	\begin{align}\label{110}
	\nonumber\int_{\Omega} h(x)(g(v))^{-\gamma}g^{\prime}(v)\varphi=& \int_{\Omega} \liminf \frac{h(x)(g(v+\epsilon \varphi))^{1-\gamma}-h(x)(g(v))^{1-\gamma}}{1-\gamma} \\
	\leq & \int_{\Omega} \nabla v \nabla \varphi -\int_{\Omega} b(x)(g(v))^{p}g^{\prime}(v)\varphi.~~~~~~~~~~~~~~
	\end{align}
	
	To end the proof of item $a)$ we will use an argument inspired by Graham-Eagle \cite{GE}. Since that $v\in \mathcal{N}_{2}$ we have	
	$$||v||^{2} - \displaystyle\int_{\Omega} b(x) (g(v))^{p}g^{\prime}(v)v-\int_{\Omega} h(x)(g(v))^{-\gamma}g^{\prime}(v)v=0.$$	
	
	For $\varphi\in H_{0}^{1}(\Omega)$ and $\epsilon>0$ define $\Psi=(v+\epsilon \varphi)^{+}$. Put $$\Omega_{1}^{\epsilon}=\left\{x\in \Omega:b(x)< 0 ~\mbox{and}~v(x)+\epsilon \varphi(x)<0 \right\}.$$ 
	
	Taking $\Psi$ as a test function in \eqref{110} we have	
	$$
	\begin{array}{l}
	0\leq  \displaystyle\int_{\Omega} \nabla v \nabla \Psi -\int_{\Omega} b(x)(g(v))^{p}g^{\prime}(v)\Psi-\int_{\Omega} h(x)(g(v))^{-\gamma}g^{\prime}(v)\Psi \\
	
	=  \displaystyle\int_{\left[v+\epsilon \varphi\geq 0\right]}\nabla v \nabla (v+\epsilon \varphi) - b(x)(g(v))^{p}g^{\prime}(v)(v+\epsilon \varphi)- h(x)(g(v))^{-\gamma}g^{\prime}(v)(v+\epsilon \varphi)  \\
	
	 =  \left(\displaystyle\int_{\Omega}-\int_{\left[v+\epsilon \varphi< 0\right]}\right)\nabla v \nabla (v+\epsilon \varphi) - b(x)(g(v))^{p}g^{\prime}(v)(v+\epsilon \varphi)- h(x)(g(v))^{-\gamma}g^{\prime}(v)(v+\epsilon \varphi)\\
	
	=  ||v||^{2} -\displaystyle\int_{\Omega} b(x) (g(v))^{p}g^{\prime}(v)v-\int_{\Omega} h(x)(g(v))^{-\gamma}g^{\prime}(v)v\\
	
	 +\epsilon \left[\displaystyle\int_{\Omega }\nabla v \nabla \varphi - b(x)(g(v))^{p}g^{\prime}(v) \varphi- h(x)(g(v))^{-\gamma}g^{\prime}(v)\varphi\right]\\
	
	 -\displaystyle\int_{\left[v+\epsilon \varphi< 0\right]}\nabla v \nabla (v+\epsilon \varphi) - b(x)(g(v))^{p}g^{\prime}(v)(v+\epsilon \varphi)- h(x)(g(v))^{-\gamma}g^{\prime}(v)(v+\epsilon \varphi)\\
	
	\leq \epsilon \left[\displaystyle\int_{\Omega }\nabla v \nabla \varphi - b(x)(g(v))^{p}g^{\prime}(v) \varphi- h(x)(g(u))^{-\gamma}g^{\prime}(v)\varphi\right]\\
	
	 -\epsilon \displaystyle\int_{\left[v+\epsilon \varphi< 0\right]}\nabla v \nabla\varphi +\epsilon\displaystyle\int_{\Omega_{1}^{\epsilon}}b(x)(g(v))^{p}g^{\prime}(v) \varphi.
	\end{array}	
	$$	
	Since the measures of the domains of integration $\left[v+\epsilon \varphi< 0\right]$ and $\Omega_{1}^{\epsilon}$ tends to zero as $\epsilon \rightarrow 0$, we then divide the last expression above by $\epsilon>0$ to obtain
	$$0\leq \int_{\Omega }\nabla v \nabla \varphi - b(x)(g(v))^{p}g^{\prime}(v) \varphi- h(x)(g(v))^{-\gamma}g^{\prime}(v)\varphi,$$
	as $\epsilon \rightarrow 0$. Replacing $\varphi$ by $-\varphi$ we conclude:
	$$ \int_{\Omega }\nabla v \nabla \varphi - b(x)(g(v))^{p}g^{\prime}(v) \varphi- h(x)(g(v))^{-\gamma}g^{\prime}(v)\varphi=0,~\forall \varphi \in H_{0}^{1}(\Omega),$$
	and therefore $v$ is a solution of $(P_{A})$. 
	
	Defining $u=g(v)$ we have that $u$ is a solution of problem $(P)$.

Now, let us prove $b)$. Suppose that $v$ is a solution of the problem $(P_{A})$. We will show that $v\in C^{1,\alpha}(\overline{\Omega})$. Hence as $g\in C^{\infty}$ we obtain that $u=g(v)\in C^{1,\alpha}(\overline{\Omega})$. 
Note that the function $v$ satisfies in the sense of distributions,
\begin{equation*}
\arraycolsep=1pt 
\medmuskip = 4mu 
\left\{ {\begin{array}{*{20}{rcl}}
-\Delta v & \geq & 0~\mbox{in}~\Omega\\
v & \geq & 0~~\mbox{in}~~\Omega.\\
\end{array}}\right.
\end{equation*}

Since that $v\in H_{0}^{1}(\Omega)$ and $v\not\equiv 0$ by Lemma \ref{BN} there exists $\epsilon>0$ such that
$$v(x)\geq \epsilon d(x,\partial \Omega),~x\in \Omega.$$

Consider $\epsilon> 0$ such that 
\begin{equation}\label{12}
\epsilon d(x,\partial \Omega) < 1,
\end{equation}
for all $x\in \Omega$. Then by (\ref{D}) and Lemma \ref{L1} $(3),(8), (10)$ there exist constants $c,C>0$ and $\beta \in (0,1)$ such that
\begin{align}\label{111}
\nonumber|h(x)(g(v))^{-\gamma}g^{\prime}(v)| & \leq  h(x)(g(\epsilon d(x,\partial \Omega)))^{-\gamma} \leq h(x)C(\epsilon d(x,\partial \Omega))^{-\gamma}\\
\nonumber& \leq C c d^{\gamma-\beta}(x,\partial \Omega)d^{-\gamma}(x,\partial \Omega) \\
&= C d^{-\beta}(x,\partial \Omega)\leq C \phi_{1}^{-\beta}(x)
\end{align}
for $x\in \Omega$, and since $\beta \in (0,1)$ follows that $h(g(v))^{-\gamma}g^{\prime}(v)\in L^{1}(\Omega)$. By Lemma \ref{Ht} there exists a solution $\Psi_{1}\in C^{1,\alpha_{1}}(\overline{\Omega})$, for some $\alpha_{1}\in (0,1)$ of the problem
$$
\left\{
\arraycolsep=1pt
\medmuskip = 4mu
\begin{array}{rl}
-\Delta w & = h(x)(g(v))^{-\gamma}g^{\prime}(v)~\mbox{in}~\Omega,\\
w & > 0 ~\mbox{in}~ \Omega \quad w=0, ~\mbox{on}~ \partial \Omega.\\
\end{array}
\right.
$$

Now, let us prove that the problem
\begin{equation}\label{15}
\left\{
\arraycolsep=1pt
\medmuskip = 4mu
\begin{array}{rl}
-\Delta w & =b(x)(g(v))^{p}g^{\prime}(v)~\mbox{in}~\Omega,\\
w & > 0 ~\mbox{in}~ \Omega \quad w=0, ~\mbox{on}~ \partial \Omega,\\
\end{array}
\right.
\end{equation}
has a unique solution $\Psi_{2}\in C^{1,\alpha_{2}}(\overline{\Omega})$ for some $\alpha_{2} \in (0,1)$. 

To do this, let $\delta:=1-p\in (0,1)$ and note that from \eqref{12} and Lemma \ref{L1} $(8),(12)$ we have
$$|b(x)g^{p}(v(x))g^{\prime}(v(x))|\leq||b||_{\infty}g^{-\delta}(v(x))(g(v(x))g^{\prime}(v(x)))\leq C \phi_{1}^{-\delta}(x),$$
that is
\begin{equation}\label{14}
|b(x)g^{p}(v(x))g^{\prime}(v(x))|\leq C \phi_{1}^{-\delta}(x),
\end{equation}
for every $x\in \Omega$ and some constant $C>0$. Therefore, by Lemma \ref{Ht} we have that the problem \eqref{15} has a unique solution $\Psi_{2}\in C^{1,\alpha_{2}}(\overline{\Omega})$ for some $\alpha_{2} \in (0,1)$. 

The existence of $\Psi_{1}$ and $\Psi_{2}$ and the fact that $v$ is a solution of $(P_{A})$ implies that
$$\int_{\Omega} \nabla v \nabla \varphi=\int_{\Omega} \left[h(x)(g(v))^{-\gamma}g^{\prime}(v)+b(x)(g(v))^{p}g^{\prime}(v)\right]\varphi=\int_{\Omega} \nabla(\Psi_{1}+\Psi_{2})\nabla \varphi,$$
for every $\varphi \in H_{0}^{1}(\Omega)$ and therefore $v=\Psi_{1}+\Psi_{2}$, which implies that $v\in C^{1,\alpha}(\overline{\Omega})$, where $\alpha:=\min \left\{\alpha_{1},\alpha_{2}\right\}\in (0,1)$.

Now, suppose that $v_{1},v_{2}$ are solutions of the problem $(P_{A})$. Let $\psi_{1}(x,s)=\psi_{2}(x,s):=h(x)(g(s))^{-\gamma}g^{\prime}(s)+b(x)(g(s))^{p}g^{\prime}(s)$. By Lemma \ref{L1} $(9),(10),(11)$ follows that for each $x\in \Omega$ the function $s \longmapsto \psi_{j}(x,s)s^{-1},j=1,2,$ is decreasing on $(0,\infty)$. Moreover by \eqref{111}
$$0\leq \psi_{j}(v_{i})\leq C\phi_{1}^{-\beta}(x)+b(x)(g(v_{i}(x)))^{p}g^{\prime}(v_{i}(x)),~x\in \Omega,$$
and therefore $\psi_{j}(x,v_{i})\in L^{1}(\Omega),j=1,2,i=1,2$. We apply the Lemma \ref{M} with $u=v_{1}$ and $v=v_{2}$ to conclude that $v_{1}\leq v_{2}$ in $\Omega$. Once more applying the Lemma \ref{M} with $u=v_{2}$ and $v=v_{1}$ we have that $v_{2}\leq v_{1}$ in $\Omega$. Therefore $v_{1}=v_{2}$ and the uniqueness follows.
\end{proof}

Let us end this section by making some remarks.
\begin{remark}\begin{itemize}
		\item[a)] If \eqref{D} is satisfied, then problem $(P)$ has solution. In fact we have $h|\phi_{1}|^{1-\gamma}\leq c|\phi_{1}|^{1-\beta}\in L^{1}(\Omega)$ and by the item $a)$ of the Theorem \ref{T1} the problem $(P)$ has solution.
		\item[b)] Lazer-Mckenna \cite{LM} proved that the solution of the following semilinear problem
			$$
			\left\{
			\arraycolsep=1pt
			\medmuskip = 4mu
			\begin{array}{rl}
			-\Delta u & = h(x)u^{-\gamma}~\mbox{in}~\Omega,\\
			u>0&\mbox{in}~\Omega,~u=0~\mbox{on}~\partial \Omega,
			\end{array}
			\right.
			$$
			does not belong to $C^{1}(\overline{\Omega})$, if $0<h\in C^{\alpha}(\overline{\Omega}),\alpha \in (0,1)$ and $1<\gamma$. Note that in this case $\displaystyle \inf_{\Omega}h>0$. When $b\equiv 0$ the proof of item $b)$ of Theorem \ref{T1} applies to the semilinear case also, showing that the solution of these problem belong to $C^{1}(\overline{\Omega})$, if the assumption (\ref{D}) is satisfied. Note that in this case $\displaystyle\inf_{\Omega} h=0$, unlike of the case in \cite{LM}.
	\end{itemize}
\end{remark}
\section{Proof of Theorem.2.}
In this section we assume that $f(x,s)=-b(x)s^{22^{\ast}-1}$ with $0\leq b\in L^{\infty}(\Omega), b\not\equiv 0$. Since that the embedding $H_{0}^{1}(\Omega)\hookrightarrow L^{2^{\ast}}(\Omega)$ is not compact, the proof of the Theorem \ref{T2} can not be applied directly. To cover this difficulty we use the Brezis-Lieb Theorem (see \cite{BL}). Now the functional associated with the problem $(P_{A})$ is
$$\Phi(v)=\frac{1}{2}\int_{\Omega} |\nabla v|^{2}+\frac{1}{\gamma-1}\int_{\Omega} h(x)(g(v))^{1-\gamma}+\frac{1}{22^{\ast}}\int_{\Omega} b(x)(g(v))^{22^{\ast}}.$$

If \eqref{C} is satisfied, using the Lemmas \ref{L2}, \ref{L5} we have $\mathcal{N}_{1}\neq \emptyset$. So it follows that.
\begin{lemma}\label{LLL1}
	The functional $\Phi$ is coercive in $\mathcal{N}_{1}$
\end{lemma}
\begin{proof}
	Indeed, for every $v\in \mathcal{N}_{1}$, we have $\Phi(v)\geq \dfrac{||v||^{2}}{2},$ and hence $\Phi$ is coercive.
\end{proof}
Let us prove the Theorem \ref{T2}.
\begin{proof}\textbf{(Theorem \ref{T2})}  Suppose that $(P)$ has a solution $v_ {0}$. Taking $v_{0}$ as test function is easy to see that \eqref{C} is satisfied.

Now assume that \eqref{C} is satisfied. Define $J_1 = \displaystyle\inf_{v \in \mathcal{N}_1} \Phi(v)$ and let $\left\{v_{n}\right\}\subset \mathcal{N}_{1}$ such that $\Phi(v_{n}) \longrightarrow J_{1}.$

	By Lemma \ref{LLL1}, $J_{1}\in \mathbb{R}$ and we may assume that $\left\{v_{n}\right\}$ is bounded in $H_{0}^{1}(\Omega)$ and in $L^{2^{\ast}}(\Omega)$ and there exists $v\in H_{0}^{1}(\Omega)$ such that 
		$$
	\left\{
	\begin{array}{l}
	 v_{n}\rightharpoonup v~\mbox{in}~H_{0}^{1}(\Omega),\\
	 v_{n}\longrightarrow v~\mbox{in}~L^{s}(\Omega)~\mbox{for all}~s\in (0,2^{\ast}),\\
	 v_{n}\longrightarrow v~a.s.~\Omega.\\
	\end{array}
	\right.
	$$
	
	Similarly to the Theorem \ref{T1} we may show that $\displaystyle\int_{\Omega} h(x)(g(v))^{-\gamma}g^{\prime}(v)v<\infty$ and consequently there exists $t(u)>0$ such that $t(v)v\in \mathcal{N}_{2}$.
	By Lemma \ref{L1} $(6)$ there exists a constant $C>0$ such that
	$$\int_{\Omega} b(x)(g(v_{n}))^{22^{\ast}}=\int_{\Omega} \left[ b^{\frac{1}{2^{\ast}}}\right]^{2^{\ast}}\left[(g(v_{n}))^{2}\right]^{2^{\ast}}\leq ||b||_{\infty}K_{0}^{22^{\ast}}\int_{\Omega} |v_{n}|^{2^{\ast}}\leq C. $$
	
	Moreover $b(x)(g(v_{n}))^{22^{\ast}}\longrightarrow b(x)(g(v))^{22^{\ast}}$ a.s. in $\Omega$. Hence by Brezis-Lieb Theorem (see \cite{BL})
	\begin{equation}\label{5}
     \begin{array}{rl}
      \displaystyle\int_{\Omega} b(x)(g(v_{n}))^{22^{\ast}} & =  \displaystyle\int_{\Omega} b(x)(g(v))^{22^{\ast}}+\int_{\Omega} b(x)|(g(v_{n}))^{22^{\ast}}-(g(v))^{22^{\ast}}|+o(1)\\
       & \geq  \displaystyle\int_{\Omega} b(x)(g(v))^{22^{\ast}}+o(1).
	\end{array}
	\end{equation}
Using the inequality (\ref{5}) and the Fatou's lemma we get
$$
\begin{array}{rl}
J_{1} & =  \lim \Phi(v_{n})\\
& =\liminf \left[\frac{1}{2}\displaystyle\int_{\Omega} |\nabla v_{n}|^{2}+\frac{1}{\gamma-1}\int_{\Omega} h(x)(g(v_{n}))^{1-\gamma}+\frac{1}{22^{\ast}}\int_{\Omega} b(x)(g(v_{n}))^{22^{\ast}}\right]\\
& \geq \frac{1}{2}\displaystyle\int_{\Omega} |\nabla v|^{2}+\frac{1}{\gamma-1}\int_{\Omega} h(x)(g(v))^{1-\gamma} + \frac{1}{22^{\ast}}\int_{\Omega} b(x)(g(v))^{22^{\ast}}\\
& =\phi_{v}(1)\geq \phi_{v}(t(v))=\Phi(t(v)v) \geq J_{2}\geq J_{1},
\end{array}
$$
that is 
	$$J_{1}=\phi_{v}(1)=\Phi(v)=J_{2},$$
	and $1$ is a critical point of $\phi_{v}$. Therefore $v\in \mathcal{N}_{2}$ and
	$$||v||^{2} +\int_{\Omega} b(x) (g(v))^{22^{\ast}-1}g^{\prime}(v)v-\int_{\Omega} h(x)(g(v))^{-\gamma}g^{\prime}(v)v=0.$$
	
	Let $\varphi \in H_{0}^{1}(\Omega), \varphi\geq 0$ and $\epsilon >0$. Then similarly to the Theorem \ref{T1} we have $h(\cdotp)(g(v+\epsilon \varphi))^{1-\gamma}\in L^{1}(\Omega)$ and therefore, by Lemmas \ref{L3},\ref{L5} there exists $t(\epsilon)>0$ such that 
	$\phi_{v+\epsilon \varphi}(t(\epsilon))=\displaystyle \inf_{t>0}\phi_{v+\epsilon \varphi}(t)$ and
	$t(\epsilon)(v+\epsilon \varphi)\in \mathcal{N}_{2}$. 
	
	Since that 
	$$\Phi(v+\epsilon \varphi)=\phi_{v+\epsilon \varphi}(1)\geq \phi_{v+\epsilon \varphi}(t(\epsilon))\geq \phi_{v}(1)=\Phi(v),$$ 
	again, similar to the Theorem \ref{T1} we may show that 
	\begin{align*}
	\int_{\Omega} h(x)(g(v))^{-\gamma}g^{\prime}(v)\varphi
	\leq \int_{\Omega} \nabla v \nabla \varphi +\int_{\Omega} b(x)(g(v))^{22^{\ast}-1}g^{\prime}(v)\varphi,~~~~~~~~~~~~~~
	\end{align*}	
	for every $\varphi\geq 0$ and 	
	$$
	\arraycolsep=1pt
        \medmuskip = 4mu
	\begin{array}{rl}                                                                                                    0 \leq & ||v||^{2} +\displaystyle\int_{\Omega} b(x) (g(v))^{22^{\ast}-1}g^{\prime}(v)v-\int_{\Omega} h(x)(g(v))^{-\gamma}g^{\prime}(v)v\\
	 & +\epsilon \left[\displaystyle\int_{\Omega }\nabla v \nabla \varphi + b(x)(g(v))^{22^{\ast}-1}g^{\prime}(v) \varphi- h(x)(g(v))^{-\gamma}g^{\prime}(v)\varphi\right]\\
         & -\displaystyle\int_{\left[v+\epsilon \varphi< 0\right]}\!\!\!\!\!\!\!\!\!\!\!\!\!\!\!\!\! \nabla v \nabla (v+\epsilon \varphi) + b(x)(g(v))^{22^{\ast}-1}g^{\prime}(v)(v+\epsilon \varphi)- h(x)(g(v))^{-\gamma}g^{\prime}(v)(v+\epsilon \varphi)\\
	  \leq & \epsilon \left[\displaystyle\int_{\Omega }\nabla v \nabla \varphi + b(x)(g(v))^{22^{\ast}-1}g^{\prime}(v) \varphi- h(x)(g(v))^{-\gamma}g^{\prime}(v)\varphi\right]\\
	 & -\epsilon \displaystyle\int_{\left[v+\epsilon \varphi< 0\right]} \!\!\!\!\!\!\!\!\!\! \nabla v \nabla \varphi+b(x)(g(v))^{22^{\ast}-1}g^{\prime}(v)\varphi,
	\end{array}
	$$
	for every $\varphi \in H_{0}^{1}(\Omega)$.
	
	Since the measure of the domain of integration $\left[v+\epsilon \varphi< 0\right]$ tends to zero as $\epsilon \rightarrow 0$, we then divide the last expression above by $\epsilon>0$ to obtain
	$$0\leq \int_{\Omega }\nabla v \nabla \varphi + b(x)(g(v))g^{\prime}(v) \varphi- h(x)(g(v))^{-\gamma}g^{\prime}(v)\varphi,$$
	as $\epsilon \rightarrow 0$. Replacing $\varphi$ by $-\varphi$ we conclude:
	$$ \int_{\Omega }\nabla v \nabla \varphi + b(x)(g(v))^{22^{\ast}-1}g^{\prime}(v) \varphi- h(x)(g(v))^{-\gamma}g^{\prime}(v)\varphi=0,~\forall \varphi \in H_{0}^{1}(\Omega),$$
	and therefore $v$ is a solution of $(P_{A})$. 
	
	Defining $u=g(v)$ we have that $u$ is a solution of problem $(P)$.	
	
	To prove that the solution is unique, let us denote by 
	$$j(x,t)=-b(x)(g(t))^{22^{\ast}-1}g^{\prime}(t)+ h(x)(g(t))^{-\gamma}g^{\prime}(t),$$
	for $x\in \Omega, t>0$. Note that $j(.,t)$ is decreasing by Lemma \ref{L1} $(9),(10)$. Suppose that $v_{1}$ and $v_{2}$ are solutions from $(P_{A})$.  Then,
		$$\rVert v_{1}-v_{2}\rVert^{2}=\displaystyle \int_{\Omega} (j(x,v_{1})-j(x,v_{2}))(v_{1}-v_{2}) <0,$$
		which implies that $v_{1}=v_{2}$. Therefore the solution is unique.
	\end{proof}
\section{Appendix}
In this section we will study the stabilty of the solutions of the following problem with parameter
\begin{equation*} 
(P_{\lambda})\left\{
\begin{array}{l}
-\Delta u -\Delta (u^{2})u=h(x) u^{-\gamma} + \lambda b(x)u^{p}~\mbox{in} ~ \Omega,\\
u> 0~\mbox{in}~ \Omega,~  u(x)=0~\mbox{on}~,\partial \Omega,
\end{array}
\right.
\end{equation*}
where $\lambda \geq 0, 0<p<1$ and $0\lneq b \in L^{\infty}(\Omega)$. The main result of this section is
\begin{theorem}\label{A}
	Let $\lambda\geq 0$. Suppose that (\ref{C}) is satisfied e let $u_{\lambda}\in H_{0}^{1}(\Omega)$ the solution from $(P_{\lambda})$, which there exists by Theorem \ref{T1}. There holds the following:
	\begin{itemize}
		\item [a)] $u_{\lambda}\geq u_{0}$ for all $\lambda >0$, 
		\item [b)] $u_{\lambda}\longrightarrow u_{0}$ in $H_{0}^{1}(\Omega)$ when $\lambda \longrightarrow 0$.
	\end{itemize}
\end{theorem}

To prove Theorem \ref{A} firstly we will consider the dual problem associated to $(P_{\lambda})$, that is for $\lambda \geq 0, 0<p<1$ and $0\lneq b \in L^{\infty}(\Omega)$ consider the following family of problems dual to $(P_{\lambda})$
\begin{equation*} 
(D_{\lambda})\left\{
\begin{array}{l}
-\Delta v  =h(x) (g(v))^{-\gamma}g^{\prime}(v) + \lambda b(x)(g(v))^{p}g^{\prime}(v)~\mbox{in} ~ \Omega,\\
v  > 0~\mbox{in}~ \Omega,~~  v(x)=0~\mbox{on}~,\partial \Omega,
\end{array}
\right.
\end{equation*}
and let $\Phi_{\lambda}$ the energy functional associated to $(D_{\lambda})$. For each $\lambda\geq 0$ let us denote by
$$\mathcal{N}_{\lambda}=\left\{v\in V_{+}:||v||^{2}-\int_{\Omega} \lambda b(g(v))^{p}g^{\prime}(v)v\geq \int_{\Omega} h(x)(g(v))^{-\gamma}g^{\prime}(v)v \right\},$$
the constrained set associated to $(D_{\lambda})$.

The proof of Theorem \ref{A} is a consequence of the following lemma.
\begin{lemma}\label{AA}
	Suppose that (\ref{C}) is satisfied and let $v_{\lambda}$ the solution from $(D_{\lambda})$ obtained in the Theorem \ref{T1}. Then
	\begin{itemize}
		\item [a)] $v_{\lambda}\geq v_{0} $, where $v_{0}$ is a unique solution from $(D_{0})$,
		\item [b)] $v_{\lambda}\longrightarrow v_{0}$ in $H_{0}^{1}(\Omega)$ when $\lambda \longrightarrow 0$,
		\item [c)] $\displaystyle\lim_{\lambda \to 0}\Phi_{\lambda}(v_{\lambda})=\Phi_{0}(v_{0})>0$,
		
		\item [d)] if the conditions of Theorem \ref{T1} $b)$ are satisfied, then the function $[0,\infty )\ni \lambda \longmapsto \Phi_{\lambda}(v_{\lambda})$ is continuous and decreasing. 
	\end{itemize}
\end{lemma}
\begin{proof} First we will to prove $a)$. Taking $(v_{\lambda}-v_{0})^{-}=\max\left\{-(v_{\lambda}-v_{0}),0\right\}$ as test function we have
$$
	\arraycolsep=1pt
        \medmuskip = 4mu
	\begin{array}{rl}
	-\rVert (v_{\lambda}-& v_{0})^{-}\rVert^{2} \\
	& = \displaystyle \int_{\Omega} ((g(v_{\lambda}))^{-\gamma}g^{\prime}(v_{\lambda})-(g(v_{0}))^{-\gamma}g^{\prime}(v_{0})+\lambda b(x)(g(v_{\lambda}))^{p}g^{\prime}(v_{\lambda}))(v_{\lambda}-v_{0})^{-}\\
	&\geq \displaystyle \int_{\Omega} ((g(v_{\lambda}))^{-\gamma}g^{\prime}(v_{\lambda})-(g(v_{0}))^{-\gamma}g^{\prime}(v_{0}))(v_{\lambda}-v_{0})^{-}\\
	&=\displaystyle \int_{\left\{v_{\lambda}<v_{0}\right\}} \!\!\!\!\!\! ((g(v_{\lambda}))^{-\gamma}g^{\prime}(v_{\lambda})-(g(v_{0}))^{-\gamma}g^{\prime}(v_{0}))(v_{\lambda}-v_{0})^{-}\geq 0,
	\end{array}
	$$
	where the last inequality is holds because the function $(g(t))^{-\gamma}g^{\prime}(t),t> 0$ is decreasing (see Lemma \ref{L1} $(9),(10)$). As a consequence of the last inequality above we have $\rVert (v_{\lambda}-v_0)^{-}\rVert=0$, which implies that $v_{\lambda}\geq v_0$ in $\Omega$.

	To prove $b)$ let $\left\{\lambda_{n}\right\}\subset (0,\infty)$ such that $\lambda_{n}\to 0$ and denote $v_{\lambda_{n}}=v_{n}$. Let us show that $\left\{v_{n}\right\}$ is bounded. Indeed, since that $\left\{v_{n}\right\}\subset \mathcal{N}_{\lambda_{n}}$ we have
	\begin{equation}\label{I}
	\rVert v_{n}\rVert^{2}=\displaystyle \int_{\Omega} h(x)(g(v_{n}))^{-\gamma}g^{\prime}(v_{n})v_{n}+\lambda_{n}\displaystyle \int_{\Omega} b (g(v_{n}))^{p}g^{\prime}(v_{n})v_{n},
	\end{equation}
	and using the Lemma \ref{L1} $(4), (5), (10)$, the item $a)$ of this Lemma and Sobolev embedding we get the following inequality
	$$
	\begin{array}{rl}
	\rVert v_{n}\rVert^{2} & \leq \displaystyle \int_{\Omega} h(x)(g(v_{n}))^{1-\gamma}+\lambda_{n}\displaystyle \int_{\Omega} b(x)(g(v_{n}))^{p+1}\\
	& \leq \displaystyle \int_{\Omega} h(x) (g(v_{0}))^{1-\gamma}+\lambda_{n}\displaystyle \int_{\Omega} b(x) |v_{n}|^{p+1}\\
	& \leq \displaystyle \int_{\Omega} h(x) (g(v_0))^{1-\gamma}+\lambda_{n}C\displaystyle  \rVert v_{n}\rVert^{p+1},~~~~~~
	\end{array}
	$$
	and since that $0<p<1$ the last inequality implies that $\left\{v_{n}\right\}$ is bounded.	
	
	Now we may assume that there exists $0 \leq \psi \in H_{0}^{1}(\Omega)$ such that 
	\begin{equation}\label{convergencia}
	\left\{
	\begin{array}{l}
	v_{n}\rightharpoonup \psi ~\mbox{in}~ H_{0}^{1}(\Omega),\\
	v_{n} \to \psi ~\mbox{in}~ L^{s}(\Omega)~ \mbox{for all}~ s\in (0,2^{\ast}),\\
	v_{n}\to \psi ~ \mbox{a.s. in}~ \Omega. 
	\end{array}
	\right.
	\end{equation}
	
	Using the Fatou's lemma in (\ref{I}) we can proceed as in the proof of the Theorem \ref{T1} to show that $\psi>0$ in $\Omega$. This considerations implies that 
	$$h(x)(g(v_{n}))^{-\gamma}g^{\prime}(v_{n})(v_{n}-\psi)\to 0~\mbox{a.s.}~ \Omega,$$
	and from Lemma \ref{L1} $(4),(9), (10)$ and $v_{n}\geq v_0$ in $\Omega$ we have
	\begin{align*}
	\vert h(x)(g(v_{n}))^{-\gamma}g^{\prime}(v_{n})(v_{n}-\psi)\vert\leq h(x)(g(v_{n}))^{1-\gamma}+h(x)(g(v_{n}))^{-\gamma}g^{\prime}(v_{n})\psi\\
	\leq h(x)(g(v_0))^{1-\gamma}+h(x)(g(v_0))^{-\gamma}g^{\prime}(v_0)\psi~~~
	\end{align*}
	and $h(x)(g(v_0))^{1-\gamma}+h(x)(g(v_0))^{-\gamma}g^{\prime}(v_0)\psi\in L^{1}(\Omega)$, because $v_{0} $ is a solution from $D_{0}$. Hence by the Lebesgue Dominated Convergence Theorem we get
	\begin{equation}\label{II}
	\displaystyle \int_{\Omega} h(x)(g(v_{n}))^{-\gamma}g^{\prime}(v_{n})(v_{n}-\psi) \longrightarrow 0.
	\end{equation}

	As a consequence of (\ref{II}) and that $v_{n}$ is a solution from $D_{\lambda_{n}}$ we have
	$$
	\arraycolsep=1pt
        \medmuskip = 4mu
	\begin{array}{l}
	\displaystyle\lim_{n\to \infty}  (v_{n}, v_{n}-\psi)=\displaystyle\lim_{n\to \infty} \int_{\Omega}\nabla v_{n}\nabla (v_{n}-\psi)=\\
	 = \displaystyle\lim_{n\to \infty}\left[ \int_{\Omega} h(x)(g(v_{n}))^{-\gamma}g^{\prime}(v_{n})(v_{n}-\psi)+\displaystyle \lambda_{n}\int_{\Omega} b(x)(g(v_{n}))^{p}g^{\prime}(v_{n})(v_{n}-\psi)\right]\\
	 =  0,
	\end{array}
	$$
	and since that $v_{n}\rightharpoonup \psi$, follows that
	$$\lim_{n\to \infty}\rVert v_{n}-\psi \rVert^{2}=\lim_{n\to \infty}(v_{n},v_{n}-\psi)+\lim_{n\to \infty}(\psi,v_{n}-\psi)=0,$$	
	which implies that $v_{n}\longrightarrow \psi $ in $H_{0}^{1}(\Omega)	$ as $n\rightarrow \infty$.
	
	To finish the proof is suficient show that $\psi=v_{0}$. Indeed, note that we have the equation
	\begin{equation}\label{130}
	\displaystyle \int_{\Omega} \nabla v_{n}\nabla \varphi=\displaystyle \int_{\Omega} h(x)(g(v_{n}))^{-\gamma}g^{\prime}(v_{n})\varphi+\displaystyle \lambda_{n}\int_{\Omega} b(x)(g(v_{n}))^{p}g^{\prime}(v_{n})\varphi, 
	\end{equation}
	being satisfied for all $\varphi \in H_{0}^{1}(\Omega)$. So, by (\ref{convergencia}) and the Lemma \ref{L1} $(9),(10)$ we have that
	$$h(x)(g(v_{n}))^{-\gamma}g^{\prime}(v_{n})\varphi \longrightarrow h(x)(g(\psi))^{-\gamma}g^{\prime}(\psi)\varphi ~\mbox{a.s.}~\Omega,$$ 
	and 
	$$|h(x)(g(v_{n}))^{-\gamma}g^{\prime}(v_{n})\varphi|\leq h(x)(g(v_{0}))^{-\gamma}g^{\prime}(v_{0})|\varphi|\in L^{1}(\Omega),$$
	were we use that $v_{0}\leq v_{n}$ in $\Omega$ and $v_{0}$ is a solution from $D_{0}$.
	
	Therefore, using \eqref{130} we may apply the dominated convergence theorem to conclude that
	 $$\displaystyle \int_{\Omega} \nabla \psi\nabla \varphi= \int_{\Omega} h(x)(g(\psi))^{-\gamma}g^{\prime}(\psi)\varphi,$$
	for every $\varphi \in H_{0}^{1}(\Omega)$, that is $\psi$ is a solution from $D_{0}$. By uniqueness of solutions from $D_{0}$ we have $\psi=v_{0}$.
	
	Let us prove $c)$. To do this, note that $v_{\lambda}\geq v_{0}$ for all $\lambda>0$ and by the item $b)$ $v_{\lambda}\longrightarrow v_{0}$ in $H^{1}_{0}(\Omega)$. Thus we can proceed as in the proof of item $b)$ and apply the Lebesgue dominated convergence theorem for get that
	$$\displaystyle \lim_{\lambda \to 0}\Phi_{\lambda}(v_{\lambda})=\Phi_{0}(v_{0}).$$

	Finally we will prove $d)$. We will give the summary proof,  since it is very similar to proof of item $b)$. Let $\lambda \in [0,\infty)$ and consider $\left\{\lambda_{n}\right\}	\subset [0,\infty)$ such that $\lambda_{n}\longrightarrow \lambda $. Using (\ref{I}) follows that $\left\{v_{n}\right\}$ is bounded in $H^{1}_{0}(\Omega)$ and there exists $\psi \in H^{1}_{0}(\Omega)$ such that $\psi>0$, $v_{n}\rightharpoonup \psi$ in $H^{1}_{0}(\Omega)$ and $v_{n}\longrightarrow \psi$ in $L^{s}(\Omega)$ for all $s\in (0,2^{\ast})$. Since $v_{n}\geq v_{0}$ we may use the Lebesgue dominated convergence Theorem to conclude that
	$$\displaystyle \lim_{n\to \infty}(v_{\lambda_{n}},v_{\lambda_{n}}-\psi)= 0,$$
	which implies that $v_{\lambda_{n}}	\longrightarrow \psi$ in $H^{1}_{0}(\Omega)$. Again by Lebesgue dominated convergence theorem we get that $\psi $ is solution from $D_{\lambda}$ and by the Theorem \ref{T1} $b)$ we have $\psi=v_{\lambda}$ and consequently
	
	$$\displaystyle \lim_{n \to \infty}\Phi_{\lambda_{n}}(v_{\lambda_{n}})=\Phi_{\lambda}(v_{\lambda}).$$
	
	Therefore the function $[0,\infty )\ni \lambda \longmapsto \Phi_{\lambda}(v_{\lambda})$ is continuous. 
	
	To prove that the function $[0,\infty )\ni \lambda\longmapsto \Phi_{\lambda}(v_{\lambda})$ is decreasing consider $0\leq \lambda<\mu $. Then we have
	$$\Phi_{\lambda}(v_{\lambda})>\Phi_{\mu}(v_{\lambda})\geq \Phi_{\mu}(t_{\mu}(v_{\lambda})v_{\lambda})\geq \Phi_{\mu}(v_{\mu}) $$
	and the proof is complete.
	\end{proof}
By the Lemma \ref{AA} we have the following picture

\begin{figure}[h]
	\centering
	\begin{tikzpicture}[scale=.30]
	\draw[thick, ->] (-1, 2) -- (8, 2);
	\draw[thick, ->] (0, -2) -- (0, 7);
	\draw[thick] (0, 5) .. controls (1, 2) and (2, 0) ..(4,-1);
	\draw (8,2) node[below]{$\lambda$};
	\draw (0, 2) node[below left]{$0$};
	\draw (3,4) node[below]{$\Phi_{\lambda}(v_{\lambda})$};
	\draw (-1.9,6) node[below]{$\Phi_{0}(v_{0})$};
	\draw (3,-1.8) node[below]{\textrm{Fig. 3 }};
	\end{tikzpicture}
	\hspace{0.8cm}
\end{figure}

Now we will prove the Theorem \ref{A}.

\begin{proof} \textbf{Theorem \ref{A}}. Firstly we prove the item \textbf{a)}. Let $v_{\lambda}$ and $v_{0}$ as in the Lemma \ref{AA}. Then $u_{\lambda}=g(v_{\lambda})$, $u_{0}=g(v_{0})$ and $v_{\lambda}\geq v_{0}$ by Lemma \ref{AA} $a)$. So 
	$$u_{\lambda}=g(v_{\lambda})\geq g(v_{0})=u_{0},$$
	because the function $g(t)$ is increasing for $t\geq 0$ (see Lemma \ref{L1} $(9)$). 
	
	To prove \textbf{b)} note that $\nabla u_{\lambda} =g^{\prime}(v_{\lambda})\nabla v_{\lambda}$ for each $\lambda\geq0$ and by inequality $(x+y)^{2}\leq 2(x^{2}+y^{2})$ for $x,y\geq 0$ we get
	$$
	\begin{array}{rl}
	\displaystyle \int_{\Omega} \rvert \nabla u_{\lambda}-\nabla u_{0}\rvert^{2} & =\displaystyle \int_{\Omega} \rvert g^{\prime}(v_{\lambda})\nabla v_{\lambda}-g^{\prime}(v_{0})\nabla v_{0}\rvert^{2} \\
	& \leq \displaystyle \int_{\Omega} (g^{\prime}(v_{\lambda})\rvert \nabla v_{\lambda}-\nabla v_{0}\rvert +\rvert g^{\prime}(v_{\lambda})-g^{\prime}(v_{0})\rvert \rvert \nabla v_{0}\rvert )^{2}\\
	& \leq 2\displaystyle \int_{\Omega} (g^{\prime}(v_{\lambda}))^{2}\rvert \nabla v_{\lambda}-\nabla v_{0}\rvert^{2} +2\displaystyle \int_{\Omega} \rvert g^{\prime}(v_{\lambda})-g^{\prime}(v_{0})\rvert^{2} \rvert \nabla v_{0}\rvert^{2}\\
	& \leq 2\displaystyle \int_{\Omega} \rvert \nabla v_{\lambda}-\nabla v_{0}\rvert^{2} +2\displaystyle \int_{\Omega} \rvert g^{\prime}(v_{\lambda})-g^{\prime}(v_{0})\rvert^{2} \rvert \nabla v_{0}\rvert^{2},~~~~~~~~~~~
	\end{array}
	$$
	where we use that $g^{\prime}(t)\leq 1$ for all $t\geq 0$ (see Lemma \ref{L1}$(3)$). By Lemma \ref{AA} $b)$ we have $v_{\lambda}\longrightarrow v_{0}$ in $H_{0}^{1}(\Omega)$ as $\lambda \rightarrow 0$, therefore since $g^{\prime}(t)\leq 1$ for all $t\geq 0$ from the Lebesgue dominated convergence Theorem follows
	$$\int_{\Omega} \rvert g^{\prime}(v_{\lambda})-g^{\prime}(v_{0})\rvert^{2} \rvert \nabla v_{0}\rvert^{2}\longrightarrow 0,$$
	as $\lambda \longrightarrow 0$. This convergence together with the last inequality above implies that $u_{\lambda}\longrightarrow u_{0}$ in $H_{0}^{1}(\Omega)$ as $\lambda \rightarrow 0$.

\end{proof}

{\it\small Ricardo Lima Alves}\\{\it\small  Departamento de Matem\'atica}\\ {\it\small Universidade de Bras\'ilia }\\
{\it\small 70910-900 Bras\'ilia}\\{\it\small DF - Brasil}\\{\it\small e-mail: ricardoalveslima8@gmail.com}\vspace{1mm}\\{\it\small  Mariana Reis}\\{\it\small  Departamento de Matem\'atica}\\ {\it\small Universidade Federal da Integração Latino-Americana }\\{\it\small 85867-970 Foz do Iguaçu}\\{\it\small PR - Brasil}\\{\it\small e-mail: mariana.reis@unila.edu.br}\vspace{1mm}

\end{document}